\documentclass[a4paper,11pt]{amsart}
\usepackage{amsfonts,amssymb}

\usepackage{graphicx,tikz}
\usepackage{float}

\DeclareMathOperator{\defect}{def}

\theoremstyle{plain}

\newtheorem{thm}{Theorem}[section]
\newtheorem*{thmA}{Theorem A}
\newtheorem*{thmB}{Theorem B}
\newtheorem{lem}[thm]{Lemma}

\theoremstyle{definition}

\newtheorem{dfn}[thm]{Definition}

\newcommand{\N}{\mathbb{N}}
\newcommand{\Z}{\mathbb{Z}}
\newcommand{\h}{\mathrm{ht}}
\newcommand{\UU}{\mathcal U}
\newcommand{\GG}{\mathcal G}

\begin{document}

\title{Outer commutator words are uniformly concise}

\author{Gustavo A. Fern\'andez-Alcober}
\address{Matematika Saila, Euskal Herriko Unibertsitatea, 48080 Bilbao, Spain}
\email{gustavo.fernandez@ehu.es}

\author{Marta Morigi}
\address{Dipartimento di Matematica, Universit\`a di Bologna,
Piazza di Porta San Donato 5, 40127 Bologna, Italy}
\email{mmorigi@dm.unibo.it}

\thanks{The first author is supported by the Spanish Ministry of Science and Innovation,
grant MTM2008-06680-C02-02, partly with FEDER funds, and by the Basque
Government, grant IT-252-07.
The second author is partially supported by MIUR (Project ``Teoria dei Gruppi e applicazioni'').}

\subjclass[2000]{Primary 20F10; Secondary 12L10}

\keywords{Outer commutator words; Verbal subgroups; Ultraproducts}

\maketitle

\begin{abstract}
We prove that outer commutator words are uniformly concise, i.e.\ if
an outer commutator word $\omega$ takes $m$ different values in a group $G$, then the
order of the verbal subgroup $\omega(G)$ is bounded by a function depending only on $m$,
and not on $\omega$ or $G$.
This is obtained as a consequence of a structure theorem for the subgroup $\omega(G)$,
which is valid if $G$ is soluble, and without assuming that $\omega$ takes finitely many
values in $G$.
More precisely, there is an abelian series of $\omega(G)$, such that every section of
the series can be generated by values of $\omega$ all of whose powers are also values of
$\omega$ in that section.
For the proof of this latter result, we introduce a new representation of outer
commutator words by means of binary trees, and we use the structure of the trees to set
up an appropriate induction.
\end{abstract}

\section{Introduction}

Let $X$ be a set of symbols, to which we refer as {\em indeterminates\/}.
In group theory, a word $\omega$ over $X$ is an element of the free group having $X$ as a
free basis.
If the expression for $\omega$ involves $k$ different indeterminates, then for every group $G$,
we obtain a function from $G^k$ to $G$ by substituting group elements for the indeterminates.
Thus we can consider the {\em set\/} $G_{\omega}$ of all values taken by this function,
that is,
\[
G_{\omega} = \{ \omega(g_1,\ldots,g_k) \mid \text{$g_i\in G$ for all $i=1,\ldots,k$} \}.
\]
The subgroup generated by $G_{\omega}$ is called the {\em verbal subgroup} of $\omega$ in $G$,
and is denoted by $\omega(G)$.
We say that two words are {\em equivalent\/} if they can be transformed into each other by
simply changing the names of the indeterminates.
Obviously, equivalent words define the same set of values and the same verbal subgroup.
For this reason, we may assume if necessary that all words are defined over the countable set
$X=\{x_1,x_2,\ldots\}$, and that if $\omega$ involves $k$ indeterminates, these are given by the
symbols $x_1,\ldots,x_k$.

Words which are formed by taking commutators are particularly interesting.
Among them, we have the {\em lower central words} $\gamma_i$, on $i$ indeterminates, which are
given by
\[
\gamma_1=x_1,
\qquad
\gamma_i=[\gamma_{i-1},x_i]=[x_1,\ldots,x_i],
\quad
\text{for $i\ge 2$,}
\]
and the {\em derived words} $\delta_i$, on $2^i$ indeterminates, defined recursively by
\[
\delta_0=x_1,
\qquad
\delta_i=[\delta_{i-1}(x_1,\ldots,x_{2^{i-1}}),\delta_{i-1}(x_{2^{i-1}+1},\ldots,x_{2^i})],
\quad
\text{for $i\ge 1$.}
\]
The words $\gamma_i$ and $\delta_i$ are particular instances of {\em outer commutator words},
which are words obtained by nesting commutators, but using always {\em different indeterminates}.
Thus $[[x_1,x_2],[x_3,x_4,x_5],x_6]$ is an outer commutator word, but the Engel word
$[x_1,x_2,x_2,x_2]$ is not.

A word $\omega$ is said to be {\em concise} if, for every group $G$, the finiteness of the set
$G_{\omega}$ implies that of $\omega(G)$.
In the 1960's, Turner-Smith published a couple of papers \cite{T-S,T-S2} related to word values
and verbal subgroups, where he indicates that Philip Hall had conjectured that every word is concise,
and that Hall himself had proved this for every non-commutator word (i.e.\ a word outside the commutator
subgroup of the free group), and for lower central words.
In \cite{T-S2}, Turner-Smith showed that also derived words are concise, and
Jeremy Wilson \cite{Wil} subsequently extended this result to all outer commutator words.
On the other hand, Hall's conjecture was eventually refuted in 1989 by Ivanov, see \cite{Iva}.

If a word $\omega$ is concise, it is natural to ask whether conciseness can be expressed in a
quantitative form; more precisely, provided that $|G_{\omega}|=m$, can we bound $|\,\omega(G)|$ by
a function depending only on $m$?
The answer to this question is positive, but this does not seem to be widely known among group
theorists and, to the best of our knowledge, there is no reference in the literature containing
this result.
For this reason, we have included an appendix at the end of the paper in which we give two different
proofs of this fact.
Both proofs need the ultraproduct construction for groups, over a non-principal ultrafilter.
The first one uses {\L}o\'s's Theorem from model theory, while the second one is derived directly
from the definition of an ultraproduct, and is due to Avinoam Mann.
Note that the existence of non-principal ultrafilters is independent of the Zermelo-Fraenkel (ZF)
axioms for set theory; it can be proved by using the Axiom of Choice (but is not equivalent to it).

The ultraproduct argument only shows the existence of bounds for concise words, but it does not provide
any explicit expressions for these bounds.
In the case of the commutator $\gamma_2=[x_1,x_2]$, one can use the results bounding the order of the
derived subgroup $G'$ in terms of the breadth (maximum size of a conjugacy class) of $G$.
If $G$ contains at most $m$ commutators, it follows that:
\begin{enumerate}
\item
If $G$ is soluble, then $|G'|\le m^{\frac{1}{2}(5+\log_2 m)}$.
(P. Neumann and Vaughan-Lee, \cite{NV}.)
\item
For a general group, $|G'|\le m^{\frac{1}{2}(13+\log_2 m)}$.
(Segal and Shalev, \cite{SS}.)
\end{enumerate}
More recently, Brazil, Krasilnikov and Shumyatsky \cite{BKS} have given explicit bounds for all lower
central words and for all derived words; as a matter of fact, they find a single upper bound for this
infinite family of words, namely $(m!)^m$.
The first main result of this paper shows that an even better uniform bound applies to all outer
commutator words.

\begin{thmA}
Let $\omega$ be an outer commutator word and let $G$ be a group.
If $|G_{\omega}|=m$, then:
\begin{enumerate}
\item
If $G$ is soluble, $|\omega(G)|\le 2^{m-1}$.
\item
If $G$ is not soluble, $|\omega(G)|\le (m-1)^{m-1}$.
\end{enumerate}
\end{thmA}

We suspect that the bounds of Theorem A might be sharpened to get close to the bounds given
above for the word $\gamma_2$.
Nevertheless, an examination of the papers giving upper bounds for $|G'|$ in terms of the breadth
clearly suggests that this would be better the subject of an independent paper, devoted specifically
to this question.
For this reason, we have not attempted to obtain sharp bounds in this paper, and we have contented
ourselves with the bounds of Theorem A which, on the other hand, are quite reasonable.

Theorem A follows without much effort from the following result, which yields structural
information about the verbal subgroup $\omega(G)$ provided that it is soluble (equivalently,
that $G$ is soluble), without assuming that $G_{\omega}$ is finite.

\begin{thmB}
Let $\omega$ be an outer commutator word, and let $G$ be  a soluble group.
Then there exists a series of subgroups from $1$ to $\omega(G)$ such that:
\begin{enumerate}
\item
All subgroups of the series are normal in $G$.
\item
Every section of the series is abelian and can be generated by values of $\omega$ all of whose
powers are also values of $\omega$ in that section.
\end{enumerate}
Furthermore, the length of this series only depends on the word $\omega$ and on the derived
length of $G$.
\end{thmB}

The existence of such a series was proved in \cite{BKS} for derived words, and this particular
case is the starting point for our proof of Theorem B.
However, dealing with an arbitrary outer commutator word $\omega$ is a much more delicate matter,
since one has to keep control of the nesting of commutators in $\omega$, and then there might be
problems such as the commutator of two values of $\omega$ not being necessarily a value of $\omega$
(contrary to the case of derived words).
Our approach to the general case is geometric: we associate a labelled binary tree to every outer
commutator word, a tree which reflects clearly the structure of the word, and which makes it easy
to compare any two outer commutator words.
Then the argument proceeds by measuring, with the help of the tree, how distant the word in question
is from being a derived word, and using induction on this distance.
The tree of an outer commutator word is introduced in Section \ref{tree}, together with some related
concepts that will be needed, and the proofs of Theorems A and B are postponed to Section \ref{proofs}.

We would like to remark that our proof of Theorem A is independent of and provides an alternative to
the proof of the conciseness of outer commutator words given by Wilson in \cite{Wil}.
Wilson's argument is rather intricate and difficult to follow, and our geometric method provides a
proof which, we honestly believe, is much easier to understand.
Also, Wilson's proof goes by way of contradiction, and consequently he does not obtain any explicit
bounds.
On the other hand, notice that our proof lies within ZF, contrary to the proof of the existence of
bounds for concise words via ultraproducts.
To end this introduction, let us say that we are highly convinced that both the `tree method' introduced
in this paper and Theorem B may prove important tools for addressing other problems related to outer
commutator words.

\section{The tree of an outer commutator word}
\label{tree}

As already mentioned, a fundamental device for the proof of Theorem B is to associate a labelled
binary tree to every outer commutator word.
For this purpose, we give a recursive and more formal definition of outer commutator words, and we
use the same recursion to introduce the height and the labelled tree of such a word.
In the following, we say that two words $\alpha$ and $\beta$ are {\em disjoint\/} if the sets of
indeterminates appearing in the two words are disjoint.

\begin{dfn}
The set of outer commutator words, and the height and the labelled tree of an outer commutator word,
are defined recursively as follows:
\begin{enumerate}
\item
An indeterminate is an outer commutator of height $0$, and its tree is an isolated vertex,
labelled with the name of the indeterminate.
\item
If $\alpha$ and $\beta$ are disjoint outer commutator words, then also $\omega=[\alpha,\beta]$
is an outer commutator word.
The height $\h(\omega)$ of the word $\omega$ is taken to be the maximum of the heights of $\alpha$
and $\beta$ plus $1$, and the tree of $\omega$ is obtained by adding a new vertex with label $\omega$
and connecting it to the vertices labelled $\alpha$ and $\beta$ of the corresponding trees of these words.
\end{enumerate}
\end{dfn}

The tree of an outer commutator word $\omega$ provides a visual way of reading how $\omega$
is constructed by nesting commutators, easier than writing the actual expression of $\omega$ by
using commutator brackets.
We draw these trees by going downwards whenever we form a new commutator, so that the vertex
with label $\omega$ is placed at the root of the tree.
Every vertex $v$ is labelled with an outer commutator word, which we denote by $\omega_v$.
Note that the indeterminates correspond exactly to the vertices of degree $1$.
Also, the height of $\omega$ coincides with the height of the tree, that is, the largest distance
from the root to another vertex of the tree (which will be necessarily labelled by an indeterminate).
For example, the following are the trees for the words $\gamma_4$ and $\delta_3$:
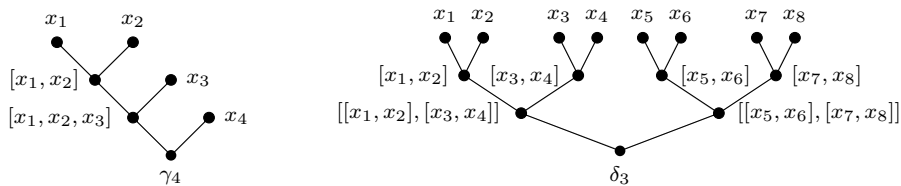
\begin{figure}[H]
\centering
\begin{tikzpicture}[level distance=5mm]
\tikzstyle{level 1}=[sibling distance=10mm]
\tikzstyle{level 2}=[sibling distance=10mm]
\tikzstyle{level 3}=[sibling distance=10mm]
\coordinate (root)[grow=up, fill] circle (2pt)
child {[fill] circle (2pt)}
child {[fill] circle (2pt)
       child {[fill] circle (2pt)}
       child {[fill] circle (2pt)
               child {[fill] circle (2pt)}
               child {[fill] circle (2pt)}
             }
      };
\node[below=2pt] at (root) {\scriptsize $\gamma_4$};
\node[left=4pt] at (root-2) {\scriptsize $[x_1,x_2,x_3]$};
\node[right=2pt] at (root-1) {\scriptsize $x_4$};
\node[left=2pt] at (root-2-2) {\scriptsize $[x_1,x_2]$};
\node[right=2pt] at (root-2-1) {\scriptsize $x_3$};
\node[above=2pt] at (root-2-2-2) {\scriptsize $x_1$};
\node[above=2pt] at (root-2-2-1) {\scriptsize $x_2$};
\end{tikzpicture}
\qquad
\begin{tikzpicture}[level distance=5mm]
\tikzstyle{level 1}=[sibling distance=26mm]
\tikzstyle{level 2}=[sibling distance=15mm]
\tikzstyle{level 3}=[sibling distance=5mm]
\coordinate (root)[grow=up, fill] circle (2pt)
child {[fill] circle (2pt)
       child {[fill] circle (2pt)
               child {[fill] circle (2pt)}
               child {[fill] circle (2pt)}
             }
       child {[fill] circle (2pt)
               child {[fill] circle (2pt)}
               child {[fill] circle (2pt)}
             }
      }
child {[fill] circle (2pt)
       child {[fill] circle (2pt)
               child {[fill] circle (2pt)}
               child {[fill] circle (2pt)}
             }
       child {[fill] circle (2pt)
               child {[fill] circle (2pt)}
               child {[fill] circle (2pt)}
             }
      };
\node[below=2pt] at (root) {\scriptsize $\delta_3$};
\node[left=4pt] at (root-2) {\scriptsize $[[x_1,x_2],[x_3,x_4]]$};
\node[right=3pt] at (root-1) {\scriptsize $[[x_5,x_6],[x_7,x_8]]$};
\node[left=2pt] at (root-2-2) {\scriptsize $[x_1,x_2]$};
\node[left=3pt] at (root-2-1) {\scriptsize $[x_3,x_4]$};
\node[right=3pt] at (root-1-2) {\scriptsize $[x_5,x_6]$};
\node[right=2pt] at (root-1-1) {\scriptsize $[x_7,x_8]$};
\node[above=2pt] at (root-2-2-2) {\scriptsize $x_1$};
\node[above=2pt] at (root-2-2-1) {\scriptsize $x_2$};
\node[above=2pt] at (root-2-1-2) {\scriptsize $x_3$};
\node[above=2pt] at (root-2-1-1) {\scriptsize $x_4$};
\node[above=2pt] at (root-1-2-2) {\scriptsize $x_5$};
\node[above=2pt] at (root-1-2-1) {\scriptsize $x_6$};
\node[above=2pt] at (root-1-1-2) {\scriptsize $x_7$};
\node[above=2pt] at (root-1-1-1) {\scriptsize $x_8$};
\end{tikzpicture}
\caption{The trees of the words $\gamma_4$ and $\delta_3$.}
\end{figure}
More generally, the full tree of height $h$ corresponds to the derived word $\delta_h$.

All labels of the tree of an outer commutator word are completely determined,
up to equivalence, by the tree itself (as a graph without labels): given the tree,
we only need to associate an indeterminate to every vertex of degree $1$, and then proceed
downwards by labeling each vertex with the commutator of the labels of its immediate
ascendants.

\vspace{5pt}

Observe that, if $\omega=[\alpha,\beta]$ is an outer commutator word, then the verbal subgroup
$\omega(G)$ coincides with the commutator subgroup $[\alpha(G),\beta(G)]$.

\vspace{5pt}

If $\omega$ is an outer commutator word, then the set $G_{\omega}$ is clearly invariant under conjugation by elements of $G$. We remark that
$G_{\omega}$ is not a subgroup in general; however, it has the following property.
\begin{lem}\label{sym}
Let $\omega$ be an outer commutator word. Then $G_{\omega}$ is symmetric, that is, $x\in G_{\omega}$ implies that $x^{-1}\in G_{\omega}$.
\end{lem}

\begin{proof}
We use induction on the height of $\omega$. If $\omega=x_1$ then the result is true.
Now assume that $\omega=[\alpha,\beta]$, where $\alpha,\beta$ are outer commutator words whose height is smaller than $\h(\omega)$.
An element of $G_{\omega}$ is of the form $[y,z]$, with $y\in G_{\alpha},z\in G_{\beta}$.
Then $[y,z]^{-1}=[y^z,z^{-1}]$, where $y^z\in G_{\alpha}$ because $G_{\alpha}$ is invariant under conjugation, and $z^{-1}\in G_{\beta}$
by induction. So $[y,z]^{-1}\in G_{\omega}$, as we wanted to prove.
\end{proof}

In the context of outer commutator words, in order to simplify the writing of words, it is
convenient to reinterpret expressions such as $[\alpha,\alpha]$ (which is $1$ in the free
group to which $\alpha$ belongs), by replacing the second $\alpha$ by an equivalent word whose
set of indeterminates is disjoint from that of the first $\alpha$.
More generally, we apply the same idea to every commutator $[\alpha,\beta]$ in which $\alpha$
and $\beta$ have some indeterminate in common, so that $[\alpha,\beta]$ is a well-defined outer
commutator word up to equivalence.
Allowing this notation, the derived words can be defined by $\delta_0=x_1$ and
$\delta_i=[\delta_{i-1},\delta_{i-1}]$ for $i\ge 1$, and the lower central words by $\gamma_1=x_1$
and $\gamma_i=[\gamma_{i-1},\gamma_1]$ for $i\ge 2$.
Also, the tree corresponding to the word $[\gamma_3,\gamma_3]$ is the following:
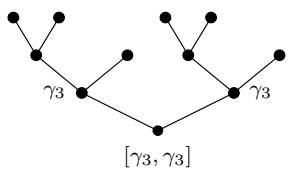
\begin{figure}[H]
\begin{tikzpicture}[level distance=5mm]
\tikzstyle{level 1}=[sibling distance=20mm]
\tikzstyle{level 2}=[sibling distance=12mm]
\tikzstyle{level 3}=[sibling distance=6mm]
\coordinate (root)[grow=up, fill] circle (2pt)
child {[fill] circle (2pt)
       child {[fill] circle (2pt)}
       child {[fill] circle (2pt)
               child {[fill] circle (2pt)}
               child {[fill] circle (2pt)}
             }
      }
child {[fill] circle (2pt)
       child {[fill] circle (2pt)}
       child {[fill] circle (2pt)
               child {[fill] circle (2pt)}
               child {[fill] circle (2pt)}
             }
      };
\node[below=2pt] at (root) {\scriptsize $[\gamma_3,\gamma_3]$};
\node[left=2pt] at (root-2) {\scriptsize $\gamma_3$};
\node[right=2pt] at (root-1) {\scriptsize $\gamma_3$};
\end{tikzpicture}
\caption{The tree of the outer commutator word $[\gamma_3,\gamma_3]$.}
\end{figure}

We note that the vertices of the tree are naturally positioned in levels.
More formally, we have the following.

\begin{dfn}
Let $v$ be a vertex of the tree of an outer commutator word $\omega$ of height $h$.
We say that $v$ is in the {\em $i$-th level} of the tree if it lies at distance $h-i$
from the root of the tree.
\end{dfn}

Thus the upmost level will be level $0$ and the root will be at level $h$, but note
that a vertex $v$ at level $i$ is not necessarily labelled with a word $\omega_v$ of
height $i$, it might even happen that $\omega_v$ is an indeterminate.

It is also useful to associate a {\em companion} vertex to each vertex of the tree
different from the root, defined as follows.

\begin{dfn}
Let $p$ be a vertex of the tree of an outer commutator word $\omega$, different from
the root, and let $u$ be the immediate descendant of $p$.
Then the {\em companion} of $p$ is the only other vertex $q$ of the tree which has $u$
as an immediate descendant.
\end{dfn}

It is clear that companion vertices lie on the same level of the tree.

\medskip

As said in the introduction, we will prove Theorem B for a general outer commutator word
$\omega$ by induction on the `distance' of $\omega$ to the closest derived word.
We make this notion of distance precise in the following definition.

\begin{dfn}
Let $\omega$ be an outer commutator word of height $h$.
Then the {\em defect\/} of $\omega$, which is denoted by $\defect\omega$, is defined as
\[
\defect\omega = 2^{h+1} - 1 - V,
\]
where $V$ is the number of vertices of the tree of $\omega$.
\end{dfn}

So, if the height of $\omega$ is $h$, then the defect is the number of vertices that need to be
added to the tree of $\omega$ in order to get the tree of $\delta_h$.
Thus the defect is $0$ if and only if $\omega$ is a derived word, and we have $\defect\gamma_4=8$
and $\defect[\gamma_3,\gamma_3]=4$.

\vspace{5pt}

Let now $\varphi$ and $\omega$ be two words, and let $F$ be the free group to which $\varphi$ belongs.
We say that $\varphi$ is {\em $\omega$-valued\/} if $\varphi\in F_{\omega}$.
If this is the case, then we have $G_{\varphi}\subseteq G_{\omega}$ for every group $G$, and in
particular $\varphi(G)\le \omega(G)$.
For example, $\delta_2$ is $\gamma_3$-valued, but not conversely.

\begin{dfn}
Let $\varphi$ and $\omega$ be two outer commutator words.
Then:
\begin{enumerate}
\item
We say that $\omega$ is a {\em constituent\/} of $\varphi$ if $\omega$ is, up to equivalence,
the label of a vertex in the tree of $\varphi$.
\item
We say that $\varphi$ is an {\em extension\/} of $\omega$, or that $\omega$ is a
{\em restriction\/} of $\varphi$, if the tree of $\varphi$ is an upward extension of the tree
of $\omega$ (simply as a tree, without labels).
\end{enumerate}
\end{dfn}

Thus, in order to get an extension of $\omega$, we only need to draw new binary trees at some of the
vertices which are labelled by indeterminates in the tree of $\omega$.
Equivalently, a restriction of $\omega$ is obtained by selecting a number of vertices and erasing all
branches lying on top of these vertices in the tree of $\omega$.
\begin{figure}[H]
\begin{tikzpicture}[level distance=5mm]
\tikzstyle{level 1}=[sibling distance=26mm]
\tikzstyle{level 2}=[sibling distance=13mm]
\tikzstyle{level 3}=[sibling distance=7mm]
\tikzstyle{level 4}=[sibling distance=4mm]
\coordinate (root)[grow=up, fill] circle (2pt)
child {[fill] circle (2pt)
       child {[fill] circle (2pt)
               child {[fill] circle (2pt)}
               child {[fill] circle (2pt)
                      child[gray!50] {[fill] circle (2pt)}
                      child[gray!50] {[fill] circle (2pt)}
                     }
             }
       child {[fill] circle (2pt)
               child {[fill] circle (2pt)
                      child[gray!50] {[fill] circle (2pt)}
                      child[gray!50] {[fill] circle (2pt)}
                     }
               child {[fill] circle (2pt)
                      child[gray!50] {[fill] circle (2pt)}
                      child[gray!50] {[fill] circle (2pt)}
                     }
             }
      }
child {[fill] circle (2pt)
       child {[fill] circle (2pt)
               child[gray!50] {[fill] circle (2pt)}
               child[gray!50] {[fill] circle (2pt)
                      child {[fill] circle (2pt)}
                      child {[fill] circle (2pt)}
                     }
             }
       child {[fill] circle (2pt)
               child {[fill] circle (2pt)}
               child {[fill] circle (2pt)
                      child {[fill] circle (2pt)}
                      child {[fill] circle (2pt)}
                     }
             }
      };
\node[below=2pt] at (root) {\scriptsize $[\gamma_4,\delta_2]$};
\draw[fill] (root-2-1) circle (2pt);
\draw[fill] (root-1-2-1) circle (2pt);
\draw[fill] (root-1-2-2) circle (2pt);
\draw[fill] (root-1-1-2) circle (2pt);
\end{tikzpicture}
\caption{\label{extensionfigure} An extension of $[\gamma_4,\delta_2]$.}
\end{figure}

In Figure~\ref{extensionfigure}, the black tree represents the word $\omega=[\gamma_4,\delta_2]$,
and the extension of $\omega$ which is obtained by adding the grey trees is
$\varphi=[[\gamma_3,\gamma_3],[\delta_2,\gamma_3]]$.
Without having to check the commutator structure of these two words, the trees show that $\varphi$
is $\omega$-valued.
On the other hand, observe that the derived word $\delta_h$ is an extension of all words of height
less than or equal to $h$.

\vspace{5pt}

The following lemma is straightforward.

\begin{lem}
\label{extension}
Let $\varphi$ and $\omega$ be two outer commutator words.
Then:
\begin{enumerate}
\item
If $\omega$ is a constituent of $\varphi$, then $\varphi(G)\le \omega(G)$.
\item
If $\varphi$ is an extension of $\omega$, then $\varphi$ is $\omega$-valued.
\end{enumerate}
\end{lem}

\section{Proof of Theorems A and B}
\label{proofs}

Before proceeding to the proof of Theorem B, we need some lemmas.
First, we need to introduce the following concept.

\begin{dfn}
Let $T$ be the tree associated to an outer commutator word $\omega$.
A subset $S$ of vertices of $T$ is called a {\em section\/} of $T$ if $S$ is maximal (with
respect to inclusion) subject to the condition that $S$ does not contain two vertices which
are one a descendant of the other.
Equivalently, in terms of labels, this means that every indeterminate involved in $\omega$
appears in exactly one word $\omega_v$ with $v\in S$.
\end{dfn}

Visually, taking a section is nothing but cutting the tree from side to side.
\begin{figure}[H]
\begin{tikzpicture}[level distance=5mm]
\tikzstyle{level 1}=[sibling distance=26mm]
\tikzstyle{level 2}=[sibling distance=13mm]
\tikzstyle{level 3}=[sibling distance=7mm]
\tikzstyle{level 4}=[sibling distance=4mm]
\coordinate (root)[grow=up, fill] circle (2pt)
child {[fill] circle (2pt)
       child {[fill] circle (2pt)
               child {[fill] circle (2pt)}
               child {[fill] circle (2pt)}
             }
       child {[fill] circle (2pt)
               child {[fill] circle (2pt)}
               child {[fill] circle (2pt)}
             }
      }
child {[fill] circle (2pt)
       child {[fill] circle (2pt)
               child {[fill] circle (2pt)}
               child {[fill] circle (2pt)
                      child {[fill] circle (2pt)}
                      child {[fill] circle (2pt)}
                     }
             }
       child {[fill] circle (2pt)
               child {[fill] circle (2pt)}
               child {[fill] circle (2pt)
                      child {[fill] circle (2pt)}
                      child {[fill] circle (2pt)}
                     }
             }
      };
\node[below=2pt] at (root) {\scriptsize $[\gamma_3,\gamma_3,\delta_2]$};
\node[transparent] (a) at (root-2-2-2) {\scriptsize aa}
node[transparent] (b) at (root-2-2-1) {\scriptsize a}
node[transparent] (c) at (root-2-1) {\scriptsize a}
node[transparent] (d) at (root-1-2) {\scriptsize a}
node[transparent] (e) at (root-1-1-2) {\scriptsize a}
node[transparent] (f) at (root-1-1-1) {\scriptsize aa};
\draw[dashed,rounded corners] (a.west) -- (a.north) -- (b.north) -- (c.north)
-- (d.north) -- (e.north) -- (f.north) -- (f.east) -- (f.south) -- (e.south)
-- (d.south) -- (c.south) -- (b.south) -- (a.south) -- cycle;
\node[right=3mm] at (root-1-1-1) {\scriptsize $S$};
\end{tikzpicture}
\caption{\label{section}A section of $[\gamma_3,\gamma_3,\delta_2]$.}
\end{figure}
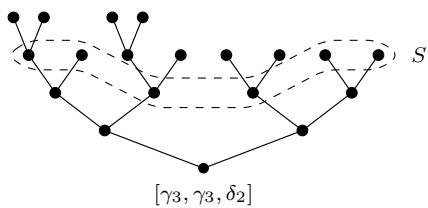

A very natural way of obtaining a section is by cutting a tree below level $i$, that is,
we consider the section $S$ containing all vertices at level $i+1$ and all the vertices
of the tree lying below level $i+1$ labelled by an indeterminate.
This is the type of section that we will use in the proof of Theorem B.
\begin{figure}[H]
\begin{tikzpicture}[level distance=5mm]
\tikzstyle{level 1}=[sibling distance=26mm]
\tikzstyle{level 2}=[sibling distance=13mm]
\tikzstyle{level 3}=[sibling distance=7mm]
\tikzstyle{level 4}=[sibling distance=4mm]
\coordinate (root)[grow=up, fill] circle (2pt)
child {[fill] circle (2pt)
       child {[fill] circle (2pt)}
       child {[fill] circle (2pt)
               child {[fill] circle (2pt)}
               child {[fill] circle (2pt)
                      child {[fill] circle (2pt)}
                      child {[fill] circle (2pt)}
                     }
             }
      }
child {[fill] circle (2pt)
       child {[fill] circle (2pt)}
       child {[fill] circle (2pt)
               child {[fill] circle (2pt)}
               child {[fill] circle (2pt)
                      child {[fill] circle (2pt)}
                      child {[fill] circle (2pt)}
                     }
             }
      };
\node[below=2pt] at (root) {\scriptsize $[\gamma_4,\gamma_4]$};
\node[transparent] (a) at (root-2-2-2) {\scriptsize aa}
node[transparent] (b) at (root-2-2-1) {\scriptsize a}
node[transparent] (c) at (root-2-1) {\scriptsize a}
node[transparent] (d) at (root-1-2-2) {\scriptsize a}
node[transparent] (e) at (root-1-2-1) {\scriptsize a}
node[transparent] (f) at (root-1-1) {\scriptsize aa};
\draw[dashed,rounded corners] (a.west) -- (a.north) -- (b.north) -- (c.north)
-- (d.north) -- (e.north) -- (f.north) -- (f.east) -- (f.south) -- (e.south)
-- (d.south) -- (c.south) -- (b.south) -- (a.south) -- cycle;
\node[right=3mm] at (root-1-1) {\scriptsize $S$};
\end{tikzpicture}
\caption{\label{section}Section of $[\gamma_4,\gamma_4]$ by cutting below level $0$.}
\end{figure}
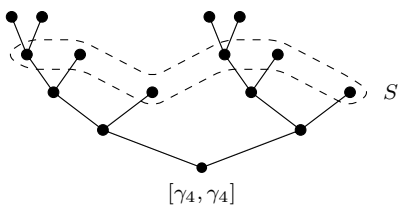

If $\omega=[\alpha,\beta]$ and $\gamma$ are two outer commutator words, then by the Three
Subgroup Lemma, we have
\[
[\omega(G),\gamma(G)]
\le
\pi^{(1)}(G) \pi^{(2)}(G),
\]
where $\pi^{(1)}=[[\alpha,\gamma],\beta]$ and $\pi^{(2)}=[\alpha,[\beta,\gamma]]$ are also outer commutator
words.
Observe that the tree of $\pi^{(1)}$ is very similar to that of $\omega$: one only needs to replace the
tree on top of the vertex labelled $\alpha$ with the tree corresponding to $[\alpha,\gamma]$.
The same happens with $\pi^{(2)}$, with $\beta$ playing the role of $\alpha$.
The following lemma is a generalization of this fact; instead of stopping at the vertices labelled
$\alpha$ and $\beta$, by iterating the process we can reach an arbitrary section of the tree.

\begin{lem}
\label{generalized 3 subgroup lemma}
Let $\omega$ be an outer commutator word, and let $T$ be the tree of $\omega$.
If $\gamma$ is another outer commutator word, then for every $v\in T$, we define $\pi^{(v)}$ to be
the word whose tree is obtained by replacing the tree of $\omega_v$ at vertex $v$ with the
tree of $[\omega_v,\gamma]$.
Then, for every section $S$ of $T$, and for every group $G$, we have
\[
[\omega(G),\gamma(G)]
\le
\prod_{v\in S} \, \pi^{(v)}(G).
\]
\end{lem}

\begin{proof}
We argue by induction on the number $n$ of vertices of $S$.
The case $n=1$ is obvious, so we assume that $n\ge 2$.
We also observe that the product $\prod_{v\in S} \, \pi^{(v)}(G)$ depends only on the subgroups
$\pi^{(v)}(G)$, for $v\in S$, and not on the order in which they appear, since all those subgroups are
normal in $G$.
Let $p$ be a vertex in $S$  which has maximum distance from the root, let $q$ be its companion vertex,
and let $u$ be the immediate descendant of $p$ and $q$.
Since each of the indeterminates involved in the word $\omega_u$ appears in exactly one of the
words $\omega_v$ with $v\in S$, it necessarily follows from the assumption about $p$ that $q\in S$.
Now let $S'$ be the section of $T$ which is obtained from $S$ by deleting $p$ and
$q$, and inserting $u$.
By applying the induction hypothesis to $S'$, we have
\begin{equation}
\label{formula for S'}
[\omega(G),\gamma(G)]
\le
\pi^{(u)}(G) \, \prod_{\substack{v\in S\\ v\ne p,q}} \, \pi^{(v)}(G).
\end{equation}
On the other hand, by the Three Subgroup Lemma,
\begin{align*}
[\omega_u(G),\gamma(G)]
&=
[\omega_p(G),\omega_q(G),\gamma(G)]
\\
&\le
[[\omega_p(G),\gamma(G)],\omega_q(G)] [\omega_p(G),[\omega_q(G),\gamma(G)]],
\end{align*}
and consequently $\pi^{(u)}(G)\le \pi^{(p)}(G)\pi^{(q)}(G)$, which completes the proof by
(\ref{formula for S'}).
\end{proof}

\begin{dfn}
Let $\omega$ be an outer commutator word, and let $G$ be a group.
A series of normal subgroups of $G$ is said to be {\em power-closed generated}
(or a {\em PCG-series\/}, for short) with respect to $\omega$ if every section $H/J$ of
the series is abelian and can be generated by values of $\omega$ in $G/J$ all of whose
powers are again values of $\omega$ in $G/J$.
\end{dfn}

It is clear that a series $H_0\le H_1 \le \cdots \le H_n$ of normal subgroups of a group
$G$ is a PCG-series with respect to $\omega$ if and only if, for every $i=1,\ldots,n$,
the quotient $H_i/H_{i-1}$ is abelian and we can choose a subset $S_i$ of $G_{\omega}$ such that:
\begin{enumerate}
\item[(P1)]
$H_i/H_{i-1}=\langle xH_{i-1}\mid x\in S_i \rangle$, i.e.\ $H_i=\langle S_i \rangle H_{i-1}$.
\item[(P2)]
$x^nH_{i-1} \in (G/H_{i-1})_{\omega}$ for every $x\in S_i$ and every $n\in\Z$.
\end{enumerate}
Furthermore, since the set of $\omega$-values is closed under conjugation, and the subgroups
in a PCG-series are normal, we may assume if necessary that $S_i$ is a normal subset of $G$.

Obviously, any PCG-series with respect to $\omega$ beginning from the trivial subgroup is
contained in $\omega(G)$, and the content of Theorem B is precisely that, starting from $1$,
we can always reach $\omega(G)$ with a PCG-series provided that $G$ is soluble.

Moreover, we note that if $\varphi$ is another outer commutator word which is $\omega$-valued,
then any PCG-series with respect to $\varphi$  is also a PCG-series with respect to $\omega$.
We will repeatedly use this fact in the sequel without further mention.

Now we state two more lemmas that we need for the proof of Theorem B.

\begin{lem}\label{joinPCG}
Let $\omega$ be an outer commutator word, let $G$ be a group, and let $K$ and $L$
be two normal subgroups of $G$.
If there are two PCG-series with respect to $\omega$ from $1$ to $K$ and from $1$
to $L$, then there is also a PCG-series from $1$ to $KL$.
\end{lem}

\begin{proof}
If $H/J$ is a normal abelian section of $G$ generated by values of $\omega$ such that all
of their powers are also values of $\omega$, then also the section $HL/JL$ has
this property.
Now the result readily follows.
\end{proof}

\begin{lem}
\label{PCG-series of commutator}
Let $\alpha$ and $\beta$ be two outer commutator words, and let $G$ be a group.
If there is a PCG-series from $K$ to $L$ in the group $G$ with respect to $\alpha$,
and if $[L,\beta(G),L]=1$, then by taking commutators with $\beta(G)$ we obtain a
PCG-series from $[K,\beta(G)]$ to $[L,\beta(G)]$ with respect to $[\alpha,\beta]$.
In particular, if there is a PCG-series from $1$ to $\alpha(G)$ with respect to $\alpha$,
and if $[\alpha(G),\beta(G),\alpha(G)]=1$, then there is also a series from $1$ to
$[\alpha(G),\beta(G)]$ with respect to $[\alpha,\beta]$.
\end{lem}

\begin{proof}
We first note that the condition $[L,\beta(G),L]=1$ implies in particular that $[L,\beta(G)]$
is abelian, so that any section of this group is also abelian.

Let $K=H_0\le H_1 \le \cdots \le H_n=L$ be a PCG-series with respect to $\alpha$.
We fix an integer $i$ from $1$ to $n$, and choose a normal subset $S_i$ of $G$ which is contained
in $G_{\alpha}$, and which satisfies properties (P1) and (P2) above.
We claim that the set $T_i=\{[x,y]\mid x\in S_i,\ y\in G_{\beta}\}$ satisfies
(P1) and (P2) for the section $[H_i,\beta(G)]/[H_{i-1},\beta(G)]$ and the word
$[\alpha,\beta]$.
This proves the result, since $T_i$ is contained in $G_{[\alpha,\beta]}$.

First of all, since $S_i$ and $G_{\beta}$ are normal subsets of $G$, the same is
true about $T_i$.
Then $N=\langle T_i \rangle [H_{i-1},\beta(G)]$ is a normal subgroup of $G$,
and $H_i$ and $\beta(G)$ clearly commute modulo $N$.
Thus $[H_i,\beta(G)]\le N$, and property (P1) follows.
Now let $[x,y]$ be an element of $T_i$, with $x\in S_i$ and $y\in G_{\beta}$.
By using the fact that $[L,\beta(G),L]=1$, we have $[x,y]^n=[x^n,y]$ for every $n\in\Z$.
Since $S_i$ satisfies (P2), we can write $x^n=a_nb_n$, with $a_n\in H_{i-1}$ and
$b_n\in G_{\alpha}$.
Thus
\[
[x,y]^n
=
[a_nb_n,y]
=
[a_n,y]^{b_n} [b_n,y]
\equiv
[b_n,y]
\pmod{[H_{i-1},\beta(G)]},
\]
which proves that (P2) holds for $T_i$.
\end{proof}

Now we can easily see that Theorem B is true for derived words.
This fact is already proved in Lemma 3.3 of \cite{BKS}, and the proof we provide is essentially the same.
We include it here for the sake of completeness, and because the use of Lemma \ref{PCG-series of commutator}
simplifies the presentation. In
the following, $G^{(i)}$ will denote as usual the $i$-th term $\delta_i(G)$ of the
derived series of a group $G$.

\begin{thm}
\label{derivedword}
Let $G$ be a soluble group.
Then, for every $i\ge 0$, there exists a PCG-series from 1 to $G^{(i)}$ with respect to $\delta_i$.
Furthermore, if the derived length of $G$ is $d$, there is such a series of length at most
$2^{d}-2^{i}$ if $d\ge i$, or $0$ if $d\le i$.
\end{thm}

\begin{proof}
We first deal with the particular case when $G^{(i)}$ is abelian.
Let us prove, by induction on $i$, that there is a PCG-series of length $2^{i}$ from $1$ to
$G^{(i)}$ with respect to $\delta_i$.
This is obvious for $i=0$, so we assume that $i\ge 1$ and that the result holds for $i-1$.
If we apply it to the group $G'$, we obtain a PCG-series of length at most $2^{i-1}$ from $1$ to
$G^{(i)}$ with respect to $\delta_{i-1}$.
By Lemma \ref{PCG-series of commutator}, with $\alpha=\beta=\delta_{i-1}$, it follows that there
is a PCG-series of the same length from $1$ to $[G^{(i)},G^{(i-1)}]$ with respect to $\delta_i$.
On the other hand, if we use again the result for $\delta_{i-1}$, but in this case with the group
$G/G^{(i)}$, we get a PCG-series of length at most $2^{i-1}$ from $G^{(i)}$ to $G^{(i-1)}$ with
respect to $\delta_{i-1}$.
Another application of Lemma \ref{PCG-series of commutator} yields a PCG-series from
$[G^{(i)},G^{(i-1)}]$ to $[G^{(i-1)},G^{(i-1)}]=G^{(i)}$ with respect to $\delta_i$.
Now we can connect the two PCG-series with respect to $\delta_i$ that we have obtained so far,
and the induction is complete.

Let us now deal with the general case.
If $i\ge d$ there is nothing to prove, so we assume that $i<d$.
By the last paragraph, for every $j$ between $i$ and $d-1$ there is a PCG-series from $G^{(j+1)}$
to $G^{(j)}$ with respect to $\delta_j$, of length $2^{j}$.
Since $\delta_j$ is $\delta_i$-valued for $j\ge i$, by connecting these series we obtain a
PCG-series from $1$ to $G^{(i)}$ with respect to $\delta_i$ of length at most $2^{d}-2^{i}$,
as desired.
\end{proof}

We can now prove Theorem B for arbitrary outer commutator words.

\begin{proof}[Proof of Theorem B]
We concentrate on proving the existence of a PCG-series with respect to $\omega$ from 1 to
$\omega(G)$; a close examination of the proof that follows shows that the length of the PCG-series
constructed only depends on $\omega$ and on the derived length of $G$, and not on the particular
group $G$.

We argue by double induction: we first use induction on the heigth of the word $\omega$, and then,
for a fixed value of the height, induction on the defect of $\omega$.
If $\omega$ has height $0$, then $\omega=x_1$ and the result is trivially true.
Now assume that $h=\h(\omega)\ge 1 $ and that the result has been proved for any outer commutator
word whose height is less than $h$.
If $\defect(\omega)=0$ then $\omega$ is a derived word, and the result holds by
Theorem \ref{derivedword}, so we assume that $\defect(\omega)>0$.
Let us write $\omega=[\alpha,\beta]$, where $\alpha$ and $\beta$ are outer commutator words of
height smaller than $h$.
Then we have a PCG-series from $1$ to $\alpha(G)$ with respect to $\alpha$, and another one from
$1$ to $\beta(G)$ with respect to $\beta$.
If we can reduce ourselves to the case that $[\omega(G),\alpha(G)]=1 $ or that $[\omega(G),\beta(G)]=1$,
then the proof of the theorem will be complete by invoking Lemma \ref{PCG-series of commutator}.

Let $\Phi$ be the (finite) set of all outer commutator words of height $h$ which are a proper
extension of $\omega$.
By the induction hypothesis on the defect, for every $\varphi$ in $\Phi$, there is a PCG-series from $1$ to
$\varphi(G)$ with respect to $\omega$, since $\varphi$ is $\omega$-valued according to
Lemma \ref{extension}.
By using Lemma \ref{joinPCG}, we can combine the series corresponding to all different words in
$\Phi$, and get a single PCG-series whose last term $L$ contains $\varphi(G)$ for all $\varphi$ in
$\Phi$.
For the theorem to be proved, it suffices to find the desired PCG-series with respect to $\omega$
in the quotient $G/L$, and so we may assume in the remainder that $\varphi(G)=1$ for all $\varphi$
in $\Phi$.
We cannot guarantee in general that $[\omega,\alpha]$ or $[\omega,\beta]$ belong to the set $\Phi$.
However, we prove below that at least one of the subgroups $[\omega(G),\alpha(G)]$ and
$[\omega(G),\beta(G)]$ is contained in a product of verbal subgroups corresponding to words in $\Phi$,
and is consequently equal to $1$, as desired.

Let $i$ be the largest integer for which there is a vertex in the tree of $\omega$ at level $i$
with label $\delta_i$.
Note that $1\le i<h$, since $\omega$ is not a derived word.
Let $S$ be the  section of the tree of $\omega$ obtained by cutting the tree below level $i$, so
that $S$ contains all vertices at level $i+1$ and all the vertices of the tree lying below level
$i+1$ which are labelled with an indeterminate.
For every vertex $v$ in $S$, we construct a word $\omega^{(v)}$ as follows.
If the label $\omega_v$ of $v$ is not an indeterminate, then we can write $\omega_v=[\omega_p,\omega_q]$,
where $p$ and $q$ are the companion vertices at level $i$ having $v$ as immediate descendant.
By the maximality of $i$, one of these vertices is labelled with a word which is different from $\delta_i$.
For simplicity, let us assume that this happens for $q$, the vertex on the right (the argument is exactly
the same otherwise).
We define $\omega^{(v)}$ to be the word whose tree is obtained by replacing $\omega_q$ with $\delta_i$ in
the tree of $\omega$.
Thus the label of $\omega^{(v)}$ at the vertex $v$ is the commutator $[\omega_p,\delta_i]$.
On the other hand, if $\omega_v$ is an indeterminate, then $\omega^{(v)}$ is defined simply by
putting the tree corresponding to $\delta_i$ on top of the vertex $v$ in the tree of $\omega$.

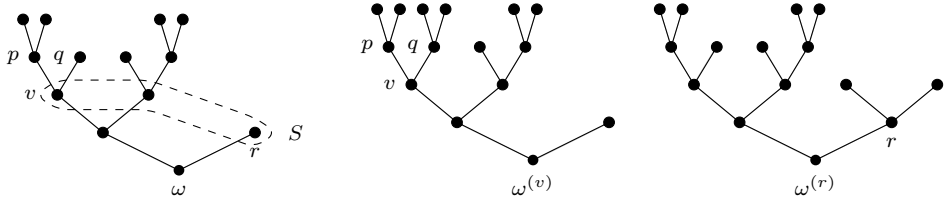
\begin{figure}[H]
\centering
\begin{tikzpicture}[level distance=5mm]
\tikzstyle{level 1}=[sibling distance=20mm]
\tikzstyle{level 2}=[sibling distance=12mm]
\tikzstyle{level 3}=[sibling distance=6mm]
\tikzstyle{level 4}=[sibling distance=3mm]
\coordinate (root)[grow=up, fill] circle (2pt)
child {[fill] circle (2pt)  }
child {[fill] circle (2pt)
       child {[fill] circle (2pt)
               child {[fill] circle (2pt)
                      child {[fill] circle (2pt)}
                      child {[fill] circle (2pt)}
                     }
              child {[fill] circle (2pt)}
             }
       child {[fill] circle (2pt)
               child {[fill] circle (2pt)}
               child {[fill] circle (2pt)
                      child {[fill] circle (2pt)}
                      child {[fill] circle (2pt)}
                     }
             }
      };
\node[transparent] (a) at (root-2-2) {\scriptsize aa}
node[transparent] (b) at (root-2-1) {\scriptsize a}
node[transparent] (c) at (root-1){\scriptsize aa};
\draw[dashed,rounded corners] (a.west) -- (a.north) -- (b.north)  --
(c.north) -- (c.east) -- (c.south) -- (b.south) -- (a.south) -- cycle;
\node[right=3mm] at (root-1) {\scriptsize $S$};
\node[left=2pt] at (root-2-2-2) {\scriptsize $p$};
\node[left=2pt] at (root-2-2-1) {\scriptsize $q$};
\node[left=4pt] at (root-2-2) {\scriptsize $v$};
\node[below=2pt] at (root-1) {\scriptsize $r$};
\node[below=2pt] at (root) {\scriptsize $\omega$};
\end{tikzpicture}
\quad
\begin{tikzpicture}[level distance=5mm]
\tikzstyle{level 1}=[sibling distance=20mm]
\tikzstyle{level 2}=[sibling distance=12mm]
\tikzstyle{level 3}=[sibling distance=6mm]
\tikzstyle{level 4}=[sibling distance=3mm]
\coordinate (root)[grow=up, fill] circle (2pt)
child {[fill] circle (2pt)  }
child {[fill] circle (2pt)
       child {[fill] circle (2pt)
               child {[fill] circle (2pt)
                      child {[fill] circle (2pt)}
                      child {[fill] circle (2pt)}
                     }
              child {[fill] circle (2pt)}
             }
       child {[fill] circle (2pt)
               child {[fill] circle (2pt)
               child {[fill] circle (2pt)}
                      child {[fill] circle (2pt)}
                      }
               child {[fill] circle (2pt)
                      child {[fill] circle (2pt)}
                      child {[fill] circle (2pt)}
                     }
             }
      };
\node[left=2pt] at (root-2-2-2) {\scriptsize $p$};
\node[left=2pt] at (root-2-2-1) {\scriptsize $q$};
\node[left=2pt] at (root-2-2) {\scriptsize $v$};
\node[below=2pt] at (root) {\scriptsize $\omega^{(v)}$};
\end{tikzpicture}
\quad
\begin{tikzpicture}[level distance=5mm]
\tikzstyle{level 1}=[sibling distance=20mm]
\tikzstyle{level 2}=[sibling distance=12mm]
\tikzstyle{level 3}=[sibling distance=6mm]
\tikzstyle{level 4}=[sibling distance=3mm]
\coordinate (root)[grow=up, fill] circle (2pt)
child {[fill] circle (2pt)
                      child {[fill] circle (2pt)}
                      child {[fill] circle (2pt)} }
                      child {[fill] circle (2pt)
       child {[fill] circle (2pt)
               child {[fill] circle (2pt)
                      child {[fill] circle (2pt)}
                      child {[fill] circle (2pt)}
                     }
              child {[fill] circle (2pt)}
             }
       child {[fill] circle (2pt)
               child {[fill] circle (2pt)}
               child {[fill] circle (2pt)
                      child {[fill] circle (2pt)}
                      child {[fill] circle (2pt)}
                     }
             }
      };
\node[below=2pt] at (root-1) {\scriptsize $r$};
\node[below=2pt] at (root) {\scriptsize $\omega^{(r)}$};
\end{tikzpicture}
\caption{\label{section}The two different cases for the construction of $\omega^{(v)}$ with $v\in S$.
Observe that $i=1$ in this example.}
\end{figure}

In any case, it is clear that $\h(\omega^{(v)})=h$ and that $\omega^{(v)}$ is a proper extension of
$\omega$, so that $\omega^{(v)}$ belongs to $\Phi$.
Consequently, we have $\omega^{(v)}(G)=1$ for every vertex $v$ in the section $S$.

On the other hand, if we apply Lemma \ref{generalized 3 subgroup lemma} to the section $S$ with
$\delta_i$ playing the role of $\gamma$, then we have
\begin{equation}
\label{omega-delta}
[\omega(G),\delta_i(G)]
\le
\prod_{v\in S} \, \pi^{(v)}(G).
\end{equation}
Here, $\pi^{(v)}$ is the word whose tree is obtained by inserting the tree of $[\omega_v,\delta_i]$
at vertex $v$ in the tree of $\omega$.
Now, it is easy to compare the two words $\omega^{(v)}$ and $\pi^{(v)}$: they look the same at all
vertices of the original tree of $\omega$, except for the vertex $v$, where $\pi^{(v)}$ has the label
$[\omega_v,\delta_i]$ and $\omega^{(v)}$ has either $[\omega_p,\delta_i]$ or $\delta_i$.
In any of the two cases, we have
\[
(\pi^{(v)})_v(G) \le (\omega^{(v)})_v(G),
\]
and then, since $\pi^{(v)}$ and $\omega^{(v)}$ have the same labels outside the tree above $v$,
also
\[
\pi^{(v)}(G) \le \omega^{(v)}(G).
\]
Since this happens for all vertices in $S$, it follows from (\ref{omega-delta}) that
$[\omega(G),\delta_i(G)]=1$.
Now, by the definition of $i$, the derived word $\delta_i$ is a constituent of either $\alpha$ or
$\beta$, and consequently either $[\omega(G),\alpha(G)]=1$ or $[\omega(G),\beta(G)]=1$.
As explained above, this completes the proof.
\end{proof}

Finally, we derive Theorem A from Theorem B by adapting the argument given by Brazil,
Krasilnikov and Shumyatsky in \cite{BKS} for the case of derived words.
We will need Dietzmann's Lemma, whose proof can we found in \cite[14.5.7]{Rob}.

\begin{lem}
\label{Dietzmann}
If $G$ is a group and $X=\{c_1,\ldots, c_n\}$ is a normal subset, then
every element $y\in \langle X\rangle$ is of the form
$y=\prod_{i=1}^nc_i^{r_i}$, for some integers $r_1,\ldots,r_n$.
\end{lem}

\begin{proof}[Proof of Theorem A]
Suppose first that $G$ is soluble.
If $\omega(G)=1$ the result is trivial, so we may assume that $\omega(G)$ is not the
identity subgroup.
By Theorem B, there is a PCG-series
\[
1=H_0 < H_1 < \cdots < H_n=\omega(G).
\]
Since $G_{\omega}$ is finite, each of the abelian quotients $H_i/H_{i-1}$ can be generated
by a finite number of values of $\omega$ all of whose powers are again values of $\omega$.
Then we can refine this PCG-series to a subnormal series
\[
1=G_0 < G_1 < \cdots < G_k=\omega(G)
\]
in which every section $G_{i}/G_{i-1}$ is a non-trivial cyclic group consisting entirely of
values of $\omega$.
Observe that, contrary to the original PCG-series, the length of this refined series may depend
on the group $G$ (more precisely, on the rank of $G$); however, this will have no effect in the
proof.
Now, for every non-trivial element $x$ in $G_{i}/G_{i-1}$, there exists
$y\in G_{\omega}\setminus \{1\}$ such that $x=yG_{i-1}$, and consequently
\[
|G_1/G_0|+|G_2/G_1|+\cdots +|G_k/G_{k-1}|\le m+k-1.
\]
Observe that $\log_2|G_i/G_{i-1}|\le |G_i/G_{i-1}|-1$ for all $i$, since $|G_i/G_{i-1}|\ge 2$.
Hence
\[
|\,\omega(G)|
=
\prod_{i=1}^k|G_i/G_{i-1}|
=
2^{\sum_{i=1}^k \log_2|G_i/G_{i-1}| }
\le
2^{\sum_{i=1}^k |G_i/G_{i-1}|-k}
\le
2^{m-1},
\]
which proves part (i) of Theorem A.

Now assume that $G$ is non-soluble.
Observe that $m\ge 3$ in this case, since otherwise $\omega(G)$ is cyclic and $G$ is soluble.
If $h$ is the height of $\omega$, then $\delta_h$ is $\omega$-valued, and the same holds
for $\delta_{h+1}$.
Let $|G_{\delta_{h+1}}|=l$.
Then $|(G/G^{(h+1)})_{\omega}|\le m-l+1$, and by the bound for the soluble case, it follows
that
\[
|\,\omega(G)/G^{(h+1)}|\le 2^{m-l}.
\]
Now we bound the order of $G^{(h+1)}$.
We claim that the order of an element $g\in G_{\delta_{h+1}}$ is at most $(m-1)(m-2)$.
Of course, we may assume $g\ne 1$.
Let us write $g=[a,b]$ with $a,b\in G_{\delta_h}$, and consider the subgroup
$H=\langle a,b \rangle$.
Let $C=C_H(a)$.
Since $a\in G_{\omega}\setminus\{1\}$, it has at most $m-1$ conjugates in $G$, and
consequently $|H:C|\le m-1$.
Now $C$ permutes the $m-1$ non-trivial values of $\omega$ in $G$, and leaves the element $a$
fixed by definition.
Thus $|C:C_C(b)|\le m-2$, and consequently $|H:Z(H)|=|H:C_H(a)\cap C_H(b)|\le (m-1)(m-2)$.
By applying Schur's Theorem \cite[10.1.4]{Rob} to $H$, it follows that the exponent of $H'$ is at
most $(m-1)(m-2)$, which proves the claim.

Let $G_{\delta_{h+1}}=\{c_0,c_1,c_2,\ldots,c_{l-1}\}$, where $c_0=1$. By Lemma \ref{sym}, the set
$G_{\delta_{h+1}}$ is symmetric, so we can assume that $c_i=c_i^{-1}$
for every $i=0,\ldots,t$ for some $0\le t\le l-1$ and  that $c_{t+2j}=c_{t+2j-1}^{-1}$ for each
$j=1,\ldots,(l-1-t)/2$ (note that $l-1-t$ is even). Since
$G_{\delta_{h+1}}$ is a normal subset of $G$, it follows from Lemma \ref{Dietzmann} that every element $w$ of $\delta_{h+1}(G)$ is
of the form $c_1^{n_1}c_2^{n_2}\cdots c_{l-1}^{n_{l-1}}$, where $1\le
n_i\le |c_i|$. Now, we have two choices for each $n_i$ with $1\le
i\le t$ (if $t\ge 1$, otherwise we have nothing to choose) and at most $(m-1)(m-2)$ choices for each product of the
form
$c_{t+2j-1}^{n_{t+2j-1}}c_{t+2j}^{n_{t+2j}}=c_{t+2j-1}^{n_{t+2j-1}-n_{t+2j}}$.
So
\[
|G^{(h+1)}|
\le
2^t[(m-1)(m-2)]^{\frac{l-1-t}2}
\le
2^t(m-1)^{l-1-t}
\le
(m-1)^{l-1},
\]
and we conclude that
\begin{align*}
|\,\omega(G)|
&=
|\,\omega(G)/G^{(h+1)}|\;|G^{(h+1)}|
\le
2^{m-l}(m-1)^{l-1}
\\
&\le
(m-1)^{m-1},
\end{align*}
since $m\ge 3$.
This completes the proof of Theorem A.
\end{proof}

\section{Appendix: Existence of bounds via ultraproducts}

In this appendix, we give two different proofs of the following result, mentioned in
the introduction.

\begin{thm}
\label{quantitative}
Let $\omega$ be a concise word.
Then, there exists a function $f:\N\rightarrow\N$ such that,
if $G$ is a group in which $|G_{\omega}|\le m$, then $|\omega(G)|\le f(m)$.
\end{thm}

For the convenience of the reader, we begin by recalling briefly the construction of
ultraproducts of groups.
To this end, we need the concept of an ultrafilter.
(See \cite{Ekl} for an account on ultraproducts from an algebraic point of view.)

\begin{dfn}
A {\em filter\/} over a non-empty set $I$ is a non-empty family $\UU$ of subsets
of $I$ such that:
\begin{enumerate}
\item
The intersection of two elements of $\UU$ also lies in $\UU$.
\item
If $P$ is in $\UU$ and $P\subseteq Q$, then also $Q$ is in $\UU$.
\item
The empty set does not belong to $\UU$.
\end{enumerate}
The filter $\UU$ is called {\em principal\/} if it consists of all supersets of a
fixed subset of $I$, and it is called an {\em ultrafilter\/} if it is maximal in
the set of all filters over $I$ ordered by inclusion.
\end{dfn}

Equivalently, a filter $\UU$ over $I$ is an ultrafilter if and only if, for every
subset $J$ of $I$, either $J\in\UU$ or $I\setminus J\in\UU$.
By (i) and (iii) above, only one of these conditions holds.

The existence of non-principal ultrafilters is independent of the Zermelo-Fraenkel
axioms for set theory.
It can be easily proved by using the Axiom of Choice, but is in fact weaker than that.
On the other hand, an ultrafilter over $I$ is non-principal if and only if it contains
all cofinite subsets of $I$.

\begin{dfn}
Let $I$ be a non-empty set, and let $\UU$ be an ultrafilter over $I$.
The {\em ultraproduct\/} modulo $\UU$ of a family $\GG=\{G_i\}_{i\in I}$ of groups is the
quotient of the cartesian product $\prod_{i\in I} \, G_i$ (i.e.\ the unrestricted direct
product) by the subgroup consisting of all tuples $(g_i)_{i\in I}$ such that the set
\[
\{i\in I\mid g_i=1\}
\]
lies in $\UU$.
We denote this ultraproduct by $\GG_{\UU}$.
\end{dfn}

Thus two tuples $(g_i)_{i\in I}$ and $(h_i)_{i\in I}$ of the cartesian product define the
same element of the ultraproduct $\GG_{\UU}$ if and only if the set of indices $i$
for which $g_i=h_i$ lies in $\UU$.
In the remainder of the paper, we use the bar notation for the image of an element or a
subset of $\prod_{i\in I}\, G_i$ in an ultraproduct.

The first proof of Theorem \ref{quantitative} that we present is based on
the following particular case of {\L}o\'s's Theorem from model theory.
(See Theorem 3.1 and Corollary 3.2 of \cite{Ekl}.)

\begin{lem}
\label{los}
Let $\GG=\{G_i\}_{i\in I}$ be a family of groups and let $\UU$ be an ultrafilter
over $I$.
Then, a sentence in the first-order language of groups holds in the ultraproduct
$\GG_{\UU}$ if and only if the set of all $i\in I$ for which the sentence holds
in $G_i$ is a member of $\UU$.
\end{lem}

Recall that the {\em width\/} of a word $\omega$ in a group $G$ is the supremum, as
$g$ ranges over the verbal subgroup $\omega(G)$, of the minimum length of all
decompositions of $g$ as a product of elements of $G_{\omega}\cup G_{\omega}^{-1}$.
Obviously, if $G$ is finite, then $\omega$ has finite width in $G$.
We may similarly speak of the width of $\omega$ over a subset $S$ of $\omega(G)$,
by taking the supremum only over elements of $S$.

\begin{proof}[First proof of Theorem \ref{quantitative}]
This proof is based on the following two facts:

(i)
For a given positive integer $m$, the property that $\omega$ takes at most $m$ values
in a group can be expressed as a sentence in the first-order language of groups.
More precisely, if $\omega$ involves $k$ indeterminates, we may use the following
formula:
\[
\exists g_{11} \ldots \exists g_{1k} \ldots \exists g_{m1} \ldots \exists g_{mk}
\
\forall x_1 \ldots \forall x_k
\
\bigvee_{i=1}^m \, \omega(x_1,\ldots,x_k)=\omega(g_{i1},\ldots,g_{ik}).
\]

(ii)
For a given positive integer $n$, the property that $\omega$ has width at most $n$ in
a group can be expressed as a sentence in the first-order language of groups.
To see this, note that this property is equivalent to every product of $n+1$ elements of
$G_{\omega}\cup G_{\omega}^{-1}$ being also a product of $n$ elements of that set.

Assume, by way of contradiction, that there is an infinite sequence $G_n$ of groups
such that $|(G_n)_{\omega}|\le m$ for all $n\in\N$ but $|\omega(G_n)|$ goes to
infinity.
Choose a non-principal ultrafilter $\UU$, and let $Q=\GG_{\UU}$.
Then, by Lemma \ref{los} and (i), we have $|Q_{\omega}|\le m$.
It follows that $|\omega(Q)|$ is finite, since $\omega$ is concise.
Then $\omega$ has finite width, say $k$, in $Q$.
By Lemma \ref{los} again, this time used together with (ii), there is a subset
$J\in\,\UU$ such that $\omega$ has width at most $k$ in $G_n$ for all $n\in J$.
Since $|(G_n)_{\omega}|\le m$, it follows that $|\omega(G_n)|\le (2m)^k$ for every
$n\in J$.
This is incompatible with the condition $\lim_{n\to\infty} \, |\omega(G_n)|=\infty$:
since $\UU$ is a non-principal ultrafilter, every cofinite subset of $\N$ has
non-empty intersection with $J$.
\end{proof}

Now we give a second proof of Theorem \ref{quantitative}, which only needs the
definition and basic properties of ultraproducts, and which is independent of
{\L}o\'s's Theorem.
This proof basically follows an argument communicated to us by Avinoam Mann.

\begin{lem}
\label{cardinality of the image}
Let $\GG=\{G_n\}_{n\in\N}$ be a family of groups, and for every $n\in\N$, let $S_n$ be
a non-empty finite subset of $G_n$.
If $\UU$ is an ultrafilter over $\N$ then the cardinality of the image of
$S=\prod_{n\in\N}\, S_n$ in the ultraproduct $\GG_{\UU}$ is given by
\begin{equation}
\label{cardinality S bar}
|\overline S| = \sup_{J\in\,\UU} \, \Big( \min_{n\in J}\, |S_n| \Big),
\end{equation}
provided that the supremum is finite, and $\overline S$ is infinite otherwise.
In particular:
\begin{enumerate}
\item
If $|S_n|\le k$ for all $n$, then $|\overline S|\le k$.
\item
If the ultrafilter $\UU$ is non-principal and $|S_n|\ge k$ for big enough $n$,
then $|\overline S|\ge k$.
\end{enumerate}
\end{lem}

\begin{proof}
Let $J$ be an arbitrary element of $\UU$, and put $m=\min_{n\in J}\, |S_n|$.
Let us prove that $|\overline S|\ge m$, which gives one of the inequalities in
(\ref{cardinality S bar}).
For every $n\in\N$, we consider $m$ elements $s_n^{(1)},\ldots,s_n^{(m)}\in S_n$, which
we take different if $n\in J$ and arbitrary if $n\not\in J$.
Let
\[
s^{(i)} = (s^{(i)}_n)_{n\in\N}, \quad \text{for every $i=1,\ldots,m$.}
\]
We claim that the images of $s^{(i)}$ and $s^{(j)}$ in $\GG_{\UU}$ are different for all
$i\ne j$.
Otherwise, the tuples $s^{(i)}$ and $s^{(j)}$ coincide on a subset $X\in\UU$, but they
are different by construction on $J\in\UU$.
Hence $J\subseteq \N\setminus X$ and, by (ii) of the definition of a filter, we also have
$\N\setminus X\in\UU$.
Thus both $X$ and $\N\setminus X$ lie in $\UU$, which is impossible since $\UU$ is an
ultrafilter.
This proves our claim, and consequently that $|\overline S|\ge m$.
Observe that this also proves that $\overline S$ is infinite if the supremum in
(\ref{cardinality S bar}) is not finite.

For the reverse inequality, put $r=\sup_{J\in\UU} \, ( \min_{n\in J}\, |S_n| )$,
and assume that $r$ is finite.
By way of contradiction, suppose that $|\overline S|\ge r+1$.
If $s^{(1)},\ldots,s^{(r+1)}$ are elements of $S$ whose images in $\GG_{\UU}$ are all
different, then for all $i,j\in\{1,\ldots,r+1\}$, $i\ne j$, the set
\[
X_{ij} = \{n\in\N \mid s_n^{(i)}\ne s_n^{(j)} \}
\]
belongs to $\UU$.
Hence the intersection $J$ of all the $X_{ij}$ is also in $\UU$.
Now observe that, if $n\in J$, then $s_n^{(1)},\ldots,s_n^{(r+1)}$ are all different and,
consequently, $|S_n|\ge r+1$.
It follows that $\min_{n\in J} \, |S_n|\ge r+1$, which is a contradiction with the
definition of $r$.

Finally, observe that (i) is obvious, and that (ii) follows because a non-principal
ultrafilter contains all cofinite subsets.
\end{proof}

If $\omega$ is a word and $\{G_i\}_{i\in I}$ is an infinite family of groups, it is not
always the case that $\omega(\prod_{i\in I}\,G_i)=\prod_{i\in I}\,\omega(G_i)$, and only
the inclusion $\subseteq$ may be guaranteed.
Our next lemma is an approximation to the reverse inclusion.

\begin{lem}
\label{omega into product}
Let $\omega$ be a word, and let $\{G_i\}_{i\in I}$ be a family of groups.
Suppose that $S_i\subseteq\omega(G_i)$ for every $i\in I$, and that the width of $\omega$
can be uniformly bounded over all the subsets $S_i$.
Then
\[
\prod_{i\in I}\, S_i \subseteq \omega(\prod_{i\in I} \, G_i).
\]
\end{lem}

\begin{proof}
Let $g=(g_i)_{i\in I}\in\prod_{i\in I}\,S_i$.
If the width of $\omega$ is at most $k$ over all the subsets $S_i$, then every $g_i$ can
be written as a product of $k$ elements $x_i^{(1)},\ldots,x_i^{(k)}$ of
$(G_i)_{\omega}\cup (G_i)_{\omega}^{-1}$.
We use these elements to define $2k$ elements of $G_i$ as follows: for every
$j=1,\ldots,k$, we put
\[
g_i^{(2j-1)}
=
\begin{cases}
x_i^{(j)}, & \text{if $x_i^{(j)}\in (G_i)_{\omega}$,}
\\
1, & \text{otherwise,}
\end{cases}
\quad
\text{and}
\quad
g_i^{(2j)}
=
\begin{cases}
1, & \text{if $x_i^{(j)}\in (G_i)_{\omega}$,}
\\
x_i^{(j)}, & \text{otherwise.}
\end{cases}
\]
Then $g_i^{(2j-1)}\in (G_i)_{\omega}$, $g_i^{(2j)}\in (G_i)_{\omega}^{-1}$ and
$g_i = g_i^{(1)} \ldots g_i^{(2k)}$ for every $i\in I$.
If we put $g^{(r)}=(g_i^{(r)})_{i\in I}$ for $r=1,\ldots,2k$, it follows that
$g^{(2j-1)}\in (\prod_{i\in I}\, G_i)_{\omega}$ and
$g^{(2j)}\in (\prod_{i\in I}\, G_i)_{\omega}^{-1}$ for $j=1,\ldots,k$, and also
that $g=g^{(1)}\ldots g^{(2k)}$.
Thus $g\in \omega(\prod_{i\in I}\,G_i)$, as desired.
\end{proof}

\begin{lem}
\label{subset of bounded width}
Let $\omega$ be a word, and let $G$ be a group such that $|\omega(G)|\ge k$, where
$k$ is a positive integer.
Then, there exists a subset $S$ of $\omega(G)$ such that $|S|\ge k$ and $\omega$
has width less than $k$ over $S$.
\end{lem}

\begin{proof}
For every integer $i\ge 0$, let $T_i$ be the subset of all elements of $\omega(G)$
of (minimum) length $i$ with respect to the set of generators
$G_{\omega}\cup G_{\omega}^{-1}$.
Put $T=\cup_{i=0}^{k-1}\,T_i$.
If $T_i$ is non-empty for every $i=0,\ldots,k-1$, then $|T|\ge k$ and we may take
$S=T$.
If, on the contrary, $T_i$ is empty for some $i=0,\ldots,k-1$, then $\omega$ has width
at most $i-1$ in $G$, and then we may take $S=\omega(G)$.
\end{proof}

\begin{proof}[Second proof of Theorem \ref{quantitative}]
By way of contradiction, assume that there is a family $\{G_n\}_{n\in\N}$ of groups such that
$|(G_n)_{\omega}|\le m$ for all $n$, but nevertheless $\lim_{n\to\infty} \, |\omega(G_n)|=\infty$.
Let us fix an arbitrary positive integer $k$.
According to Lemma \ref{subset of bounded width}, if $n$ is big enough, there is a subset
$S_n$ of $\omega(G_n)$ such that $|S_n|\ge k$ and $\omega$ has width less than $k$ over $S_n$.
We complete the sequence $\{S_n\}_{n\in\N}$ by choosing the first terms equal to $1$.
Now, if $G=\prod_{n\in\N} \, G_n$ and $S=\prod_{n\in\N} \, S_n$, we have
\[
G_{\omega} = \prod_{n\in\N} \, (G_n)_{\omega},
\quad
\text{and}
\quad
S \subseteq \omega(G),
\]
where the last inclusion follows from Lemma \ref{omega into product}.
Consider now a non-principal ultrafilter $\UU$ over $\N$, and let $Q=\GG_{\UU}$ be the
corresponding ultraproduct.
Then $Q_{\omega}=\overline{(G_{\omega})}$ and
$\omega(Q)=\overline{\omega(G)}\supseteq \overline S$.
By applying Lemma \ref{cardinality of the image}, we obtain that $|Q_{\omega}|\le m$ and
$|\omega(Q)|\ge k$.
Since $k$ is arbitrary, we get $|\omega(Q)|=\infty$, which is a contradiction, since the word
$\omega$ is concise.
\end{proof}

\noindent
\textit{Acknowledgments.}
We want to thank Avinoam Mann for communicating to us the argument we have used in the second
proof of Theorem \ref{quantitative}, and for giving us permission to include it in this paper.
On the other hand, we want to express our gratitude to the University of the Basque Country
and the University of Padova for their hospitality while this work was carried out.

\end{document}